\documentclass[12pt]{amsart}
\usepackage{amsfonts,amssymb,amscd,amsmath,enumerate,verbatim}
\usepackage[latin1]{inputenc}
\usepackage{amscd}
\usepackage{latexsym}
\usepackage{mathptmx}
\usepackage{multicol}
\input xy
\xyoption{all}
\def\NZQ{\mathbb}               

\def\RR{{\NZQ R}}

\def\frk{\mathfrak}               

\def\Phi{{\frk N}}
\def\eb{{\bold e}}

\def\opn#1#2{\def#1{\operatorname{#2}}} 
\opn\chara{char} \opn\length{\ell} \opn\pd{pd} \opn\rk{rk}
\opn\projdim{proj\,dim} \opn\injdim{inj\,dim} \opn\rank{rank}
\opn\depth{depth} \opn\grade{grade} \opn\height{height}
\opn\embdim{emb\,dim} \opn\codim{codim}

\opn\Tr{Tr} \opn\bigrank{big\,rank}
\opn\superheight{superheight}\opn\lcm{lcm}
\opn\trdeg{tr\,deg}
\opn\reg{reg} \opn\lreg{lreg} \opn\ini{in} \opn\lpd{lpd}
\opn\size{size}\opn{\mult}{mult}
\opn\div{div} \opn\Div{Div} \opn\cl{cl} \opn\Cl{Cl}
\opn\Spec{Spec} \opn\Supp{Supp} \opn\supp{supp} \opn\Sing{Sing}
\opn\Ass{Ass} \opn\Min{Min}
\opn\Ann{Ann} \opn\Rad{Rad} \opn\Soc{Soc}
\opn\Syz{Syz} \opn\Im{Im} \opn\Ker{Ker} \opn\Coker{Coker}
\opn\Am{Am} \opn\Hom{Hom} \opn\Tor{Tor} \opn\Ext{Ext}
\opn\End{End} \opn\Aut{Aut} \opn\id{id} \opn\ini{in}

\opn\nat{nat}
\opn\pff{pf}
\opn\Pf{Pf} \opn\GL{GL} \opn\SL{SL} \opn\mod{mod} \opn\ord{ord}
\opn\Gin{Gin}
\opn\Hilb{Hilb}\opn\adeg{adeg}\opn\std{std}\opn\ip{infpt}
\opn\Pol{Pol}
\opn\sat{sat}
\opn\Var{Var}
\opn\Gen{Gen}

\opn\aff{aff} \opn\con{conv} \opn\relint{relint} \opn\st{st}
\opn\lk{lk} \opn\cn{cn} \opn\core{core} \opn\vol{vol}
\opn\link{link} \opn\star{star}
\opn\gr{gr}

\def\Cc{{\mathcal C}}
\def\Oc{{\mathcal O}}

\def\Jc{{\mathcal J}}

\def\pot#1#2{#1[\kern-0.28ex[#2]\kern-0.28ex]}

\opn\dirlim{\underrightarrow{\lim}}
\opn\inivlim{\underleftarrow{\lim}}
\let\to=\rightarrow

\def\Implies{\ifmmode\Longrightarrow \else
        \unskip${}\Longrightarrow{}$\ignorespaces\fi}
\def\implies{\ifmmode\Rightarrow \else
        \unskip${}\Rightarrow{}$\ignorespaces\fi}
\def\iff{\ifmmode\Longleftrightarrow \else
        \unskip${}\Longleftrightarrow{}$\ignorespaces\fi}

\let\:=\colon
\newtheorem{Theorem}{Theorem}[section]

\newtheorem{Corollary}[Theorem]{Corollary}

\newtheorem{Example}[Theorem]{Example}

\newtheorem{Conjecture}[Theorem]{Conjecture}
\newtheorem{Question}[Theorem]{Question}
\let\epsilon\varepsilon
\let\phi=\varphi
\let\kappa=\varkappa
\def\qed{\ifhmode\textqed\fi
      \ifmmode\ifinner\quad\qedsymbol\else\dispqed\fi\fi}
\def\textqed{\unskip\nobreak\penalty50
       \hskip2em\hbox{}\nobreak\hfil\qedsymbol
       \parfillskip=0pt \finalhyphendemerits=0}
\def\dispqed{\rlap{\qquad\qedsymbol}}

\opn\dis{dis}
\def\pnt{{\raise0.5mm\hbox{\large\bf.}}}

\opn\Lex{Lex}

\begin{document}
\title{Chain polytopes and algebras with straightening laws}
\author{Takayuki Hibi and Nan Li}
\thanks{}
\address{Takayuki Hibi,
Department of Pure and Applied Mathematics,
Graduate School of Information Science and Technology,
Osaka University,
Toyonaka, Osaka 560-0043, Japan}
\email{hibi@math.sci.osaka-u.ac.jp}
\address{Nan Li,
Department of Mathematics,
Massachusetts Institute of Technology,
Cambridge, MA 02139, USA}
\email{nan@math.mit.edu}
\thanks{}
\begin{abstract}
It will be shown that the toric ring of the chain polytope of a finite partially ordered set
is an algebra with straightening laws on a finite distributive lattice.
Thus in particular every chain polytope possesses a regular unimodular triangulation
arising from a flag complex.
\end{abstract}
\thanks{
\hspace{5.3mm}{\bf Key words and phrases:}
algebra with straightening laws, chain polytope, partially ordered set.}
\maketitle
\section*{Introduction}
In \cite{Stanley}, the order polytope $\Oc(P)$ and the chain polytope $\Cc(P)$
of a finite poset (partially ordered set) $P$ are studied in detail from a view point of combinatorics.
Toric rings of order polytopes are studied in \cite{Hibi}.
In particular it is shown that the toric ring $K[\Oc(P)]$ of the order polytope $\Oc(P)$ is
an algebra with straightening laws (\cite[p. 124]{HibiRedBook}) on a finite distributive lattice.
In the present paper, it will be proved that the toric ring $K[\Cc(P)]$ of the chain polytope $\Cc(P)$
is also an algebra with straightening laws on a finite distributive lattice.
It then follows immediately that $\Cc(P)$ possesses a regular unimodular triangulation
arising from a flag complex.

\section{Toric rings of order polytopes and chain polytopes}
Let $P = \{x_{1}, \ldots, x_{d}\}$ be a finite poset.
For each subset $W \subset P$, we associate $\rho(W) = \sum_{i \in W}\eb_{i} \in \RR^{d}$,
where $\eb_{1}, \ldots, \eb_{d}$ are the unit coordinate vectors of $\RR^{d}$.
In particular $\rho(\emptyset)$ is the origin of $\RR^{d}$.
A {\em poset ideal} of $P$ is a subset $I$ of $P$ such that,
for all $x_{i}$ and $x_{j}$ with
$x_{i} \in I$ and $x_{j} \leq x_{i}$, one has $x_{j} \in I$.
An {\em antichain} of $P$ is a subset
$A$ of $P$ such that $x_{i}$ and $x_{j}$ belonging to $A$ with $i \neq j$ are incomparable.

Recall that the {\em order polytope} is the convex polytope $\Oc(P) \subset \RR^{d}$
which consists of those $(a_{1}, \ldots, a_{d}) \in \RR^{d}$ such that
$0 \leq a_{i} \leq 1$ for every $1 \leq i \leq d$ together with
$a_{i} \geq a_{j}$
if $x_{i} \leq x_{j}$ in $P$.
The vertices of $\Oc(P)$ is
those $\rho(I)$ such that $I$ is a poset ideal of $P$
(\cite[Corollary 1.3]{Stanley}).
The {\em chain polytope} is the convex polytope $\Cc(P) \subset \RR^{d}$
which consists of those $(a_{1}, \ldots, a_{d}) \in \RR^{d}$ such that
$a_{i} \geq 0$ for every $1 \leq i \leq d$ together with
\[
a_{i_{1}} + a_{i_{2}} + \cdots + a_{i_{k}} \leq 1
\]
for every maximal chain $x_{i_{1}} < x_{i_{2}} < \cdots < x_{i_{k}}$ of $P$.
The the vertices of $\Cc(P)$ is
those $\rho(A)$ such that $A$ is an antichain of $P$
(\cite[Theorem 2.2]{Stanley}).

Let $S = K[x_{1}, \ldots, x_{d}, t]$ denote the polynomial ring over a field $K$
whose variables are the elements of $P$ together with the new variable $t$.
For each subset $W \subset P$, we associate the squarefree monomial
$x(W) = \prod_{i \in W} x_{i} \in S$.
In particular $x(\emptyset) = 1$.
The {\em toric ring} $K[\Oc(P)]$ of $\Oc(P)$ is
the subring of $R$ generated by those monomials \ $t \cdot x({I})$ \ such that
$I$ is a poset ideal of $P$.
The toric ring $K[\Cc(P)]$ of $\Cc(P)$ is
the subring of $R$ generated by those monomials \ $t \cdot x({A})$ \ such that
$A$ is an antichain of $P$.

\section{Algebras with straightening laws}
Let $R = \bigoplus_{n \geq 0} R_{n}$ be a graded algebra over a field $R_{0} = K$.
Suppose that $P$ is a poset with an injection $\varphi : P \to R$
such that $\varphi(\alpha)$ is a homogeneous element of $R$ with $\deg \varphi(\alpha) \geq 1$
for every $\alpha \in P$.  A {\em standard monomial} of $R$ is a finite product of the form
$\varphi(\alpha_{1})\varphi(\alpha_{2}) \cdots$
with $\alpha_{1} \leq \alpha_{2} \leq \cdots$.
Then we say that $R = \bigoplus_{n \geq 0} R_{n}$
is an {\em algebra with straightening laws} on $P$ over $K$
if the following conditions are satisfied:
\begin{itemize}
\item
The set of standard monomials is a basis of $R$ as a vector space over $K$;
\item
If $\alpha$ and $\beta$ in $P$ are incomparable and if
\begin{eqnarray}
\label{relation}
\varphi(\alpha)\varphi(\beta) = \sum_{i} r_{i}
\varphi(\gamma_{i_{1}})\varphi(\gamma_{i_{2}}) \cdots, 
\end{eqnarray}
where $0 \neq r_{i} \in K$ and
$\gamma_{i_{1}} \leq \gamma_{i_{2}} \leq \cdots$, 
is the unique expression for $\varphi(\alpha)\varphi(\beta) \in R$
as a linear combination of distinct standard monomials, then
$\gamma_{i_{1}} \leq \alpha$ and $\gamma_{i_{1}} \leq \beta$ for every $i$.
\end{itemize}

We refers the reader to \cite[Chapter XIII]{HibiRedBook} for fundamental materials on
algebras with straightening laws.
The relations (\ref{relation}) are called the {\em straightening relations} of $R$.

Let $P$ be an arbitrary finite poset and $\Jc(P)$ the finite distributive lattice
(\cite[p. 252]{StanleyEC})
consisting of all poset ideals of $P$, ordered by inclusion.
The toric ring $K[\Oc(P)]$ of the order polytope $\Oc(P)$ can be a graded ring
with $\deg (t \cdot x(I)) = 1$ for every $I \in \Jc(P)$.
We then define the injection $\varphi : \Jc(P) \to K[\Oc(P)]$ by setting
$\varphi(I) = t \cdot x(I)$ for every $I \in \Jc(P)$.
One of the fundamental results obtained in \cite{Hibi} is that $K[\Oc(P)]$
is an algebra with straightening laws on $\Jc(P)$.  Its
straightening relations are
\begin{eqnarray}
\label{order}
\varphi(I)\varphi(J) = \varphi(I \cap J)\varphi(I \cup J),
\end{eqnarray}
where $I$ and $J$ are poset ideals of $P$ which are incomparable in $\Jc(P)$.

\begin{Theorem}
\label{ASLchain}
The toric ring of the chain polytope of a finite poset is
an algebra with straightening laws on a finite distributive lattice.
\end{Theorem}

\begin{proof}
Let $P$ be an arbitrary finite poset and $\Cc(P)$ its chain polytope.
The toric ring $K[\Cc(P)]$ can be a graded ring
with $\deg (t \cdot x(A)) = 1$ for every antichain $A$ of $P$.
For a subset $Z \subset P$, we write $\max(Z)$ for the set of
maximal elements of $Z$.  In particular $\max(Z)$ is an antichain of $P$.
The poset ideal of $P$ generated by
a subset $Y \subset P$ is the smallest poset ideal of $P$ which contains $Y$.

Now, we define the injection $\psi : \Jc(P) \to K[\Cc(P)]$ by setting
$\psi(I) = t \cdot x(\max(I))$ for all poset ideal $I$ of $P$.
If $I$ and $J$ are poset ideals of $P$, then
\begin{eqnarray}
\label{chain}
\psi(I)\psi(J) = \psi(I \cup J) \psi(I*J),
\end{eqnarray}
where $I*J$ is the poset ideal of $P$ generated by
$\max(I \cap J) \cap (\max(I) \cup \max(J))$.  Since
$I*J \subset I$ and $I*J \subset J$, the relations (\ref{chain}) satisfy
the condition of the straightening relations.

It remains to prove that the set of standard monomials of
$K[\Cc(P)]$ is a $K$-basis of $K[\Cc(P)]$.
It follows from \cite[Theorem 4.1]{Stanley} that the Hilbert function
(\cite[p. 33]{HibiRedBook})
of the Ehrhart ring (\cite[p. 97]{HibiRedBook}) of $\Oc(P)$
coincides with that of $\Cc(P)$.
Since $\Oc(P)$ and $\Cc(P)$ possess the integer decomposition
property (\cite[Lemma 2.1]{OhHicompressed}), the Ehrhart ring of $\Oc(P)$
coincides with $K[\Oc(P)]$ and the Ehrhart ring of $\Cc(P)$
coincides with $K[\Cc(P)]$.  Hence the Hilbert function of
$K[\Oc(P)]$ is equal to that of $K[\Cc(P)]$.
Thus the set of standard monomials of
$K[\Cc(P)]$ is the $K$-basis of $K[\Cc(P)]$, as desired.
\, \, \, \, \, \, \, \, \, \,
\, \, \, \, \, \, \, \,
\end{proof}

\section{Flag and unimodular triangulations}
The fact that $K[\Cc(P)]$ is an algebra with straightening laws guarantees that
the toric ideal of $\Cc(P)$ possesses an initial ideal
generated by squarefree quadratic monomials.
We refer the reader to \cite{OhHiquadratic} and \cite[Appendix]{OhHirootsystem} for the background
of the existence of squarefree quadratic initial ideals of toric ideals.
By virtue of \cite[Theorem 8.3]{Sturmfels}, it follows that

\begin{Corollary}
\label{ASLchain}
Every chain polytope possesses a regular unimodular triangulation arising from a flag complex.
\end{Corollary}

\section{Further questions}
Let, as before, $P$ be a finite poset and $\Jc(P)$ the finite distributive lattice
consisting of all poset ideals of $P$, ordered by inclusion.
Let $S = K[x_{1}, \ldots, x_{n},t]$ denote the polynomial ring and
$\Omega = \{ w_{I} \}_{I \in \Jc(P)}$ a set of monomials in $x_{1}, \ldots, x_{n}$
indexed by $\Jc(P)$.  We write $K[\Omega]$ for the subring of $S$ generated by those monomials
$w_{I} \cdot t$ with $I \in \Jc(P)$ and define the injection $\varphi : \Jc(P) \to K[\Omega]$
by setting $\varphi(I) = w_{I} \cdot t$ for every $I \in \Jc(P)$.

Suppose that $K[\Omega]$ is an algebra with straightening laws on $\Jc(P)$ over $K$.
We say that $K[\Omega]$ is {\em compatible} if each of its straightening relations
is of the form
$\varphi(I)\varphi(I') = \varphi(J)\varphi(J')$
such that $J \leq I \wedge I'$ and $J' \geq I \vee I'$,
where $I$ and $I'$ are poset ideals of $P$ which are incomparable in $\Jc(P)$.

Let $K[\Omega]$ and $K[\Omega']$ be compatible algebras with straightening laws
on $\Jc(P)$ over $K$.  Then we identify $K[\Omega]$ with $K[\Omega']$ if
the straightening relations of $K[\Omega]$ coincide with those of $K[\Omega']$.

Let $P^{*}$ be the {\em dual poset} (\cite[p. 247]{StanleyEC}) of a poset $P$.
The toric ring $K[\Cc(P^{*})]$ of $\Cc(P^{*})$ can be regarded as
an algebra with straightening laws on $\Jc(P)$ over $K$ in the obvious way.
Clearly each of the toric rings $K[\Oc(P)]$, $K[\Cc(P^)]$ and $K[\Cc(P^{*})]$ is
a compatible algebra with straightening laws on $\Jc(P)$ over $K$

\begin{Question}
{\em (a)}
Given a finite poset $P$, find all possible compatible algebras
with straightening laws on $\Jc(P)$ over $K$.

{\em (b)}
In particular, for which posets $P$, does there exist a unique compatible
an algebra with straightening laws on $\Jc(P)$ over $K$\,?
\end{Question}

\begin{Example}
{\em
(a)
Let $P$ be a poset of Figure $1$.  Then $K[\Oc(P)] = K[\Cc(P^)]$ and
there exists a unique a compatible algebra with straightening laws
on $\Jc(P)$ over $K$.

(b)
Let $P$ be a poset of Figure $2$.
Then there exists three compatible algebras with straightening laws on $\Jc(P)$ over $K$.
They are $K[\Oc(P)]$, $K[\Cc(P^)]$ and $K[\Cc(P^{*})]$.

(c)
Let $P$ be a poset of Figure $3$.
Then there exists nine compatible algebras with straightening laws on $\Jc(P)$ over $K$.

\newpage

$$
\begin{array}{ccccc}
  
\xy 0;/r.2pc/: 
 (8,5)*{\circ}="b";
 (-3,-10)*{\circ}="d";
 (8,-10)*{\circ}="e";
  "e"; "b"**\dir{-};
    
\endxy & \,\,\,\,\,\,\,\,\,\,\,\,\,\,\,\,&
\xy 0;/r.2pc/: (-3,5)*{\circ}="a";
 (8,5)*{\circ}="b";
 (-3,-10)*{\circ}="d";
 (8,-10)*{\circ}="e";
  "e"; "b"**\dir{-};
    "a"; "d"**\dir{-};
     "a"; "e"**\dir{-};
\endxy
 & \,\,\,\,\,\,\,\,\,\,\,\,\,\,\,\,&
\xy 0;/r.2pc/: (0,0)*{\circ}="a";
 (8,10)*{\circ}="b";
 (-8,10)*{\circ}="c";
 (-8,-10)*{\circ}="d";
 (8,-10)*{\circ}="e";
  "a"; "b"**\dir{-};
   "a"; "c"**\dir{-};
    "a"; "d"**\dir{-};
     "a"; "e"**\dir{-};
\endxy
 \\
 \,\,\,\,\,\,\,\,\,\,\,\,\,\,\,\,&\,\,\,\,\,\,\,\,\,\,\,\,\,\,\,\,&\,\,\,\,\,\,\,\,\,\,\,\,\,\,\,\,
 &\,\,\,\,\,\,\,\,\,\,\,\,\,\,\,\,&\,\,\,\,\,\,\,\,\,\,\,\,\,\,\,\,
 \\
  \text{ Figure }1 &\,\,\,\,\,\,\,\,\,\,\,\,\,\,\,\,& \text{ Figure } 2 &\,\,\,\,\,\,\,\,\,\,\,\,\,\,\,\,& \text{ Figure } 3
\end{array}$$

%
}
\end{Example}

\smallskip

\begin{Conjecture}
{\em
If $P$ is a disjoint union of chains, then
the compatible algebras with straightening laws on $\Jc(P)$ over $K$
are $K[\Oc(P)]$, $K[\Cc(P^)]$ and $K[\Cc(P^{*})]$.
}
\end{Conjecture}

{}

\end{document}